\documentclass[12pt]{amsart}
\usepackage{latexsym,amssymb,amsmath, amsfonts}
\usepackage{graphicx}
\usepackage{xcolor}

\numberwithin{equation}{section}

\newtheorem{theorem}{Theorem}[section]

\newtheorem{proposition}[theorem]{Proposition}
\newtheorem{Definition}{Definition}[section]

\newtheorem{Rem}{Remark}[section]

\newcommand\R{\mathbb{R}}

\newcommand\be{\begin{equation}}
\newcommand\ee{\end{equation}}
\newcommand\bea{\begin{eqnarray}}
\newcommand\eea{\end{eqnarray}}
\newcommand\beaa{\begin{eqnarray*}}
\newcommand\eeaa{\end{eqnarray*}}
\newcommand\bss{\begin{cases}}
\newcommand\ess{\end{cases}}

\newcommand\bR{\mathbb{R}}

\newcommand{\ld}{\lambda}

\newcommand\mI{{\mathcal{I}}}
\newcommand\mK{{\mathcal{E}}}
\begin{document}

\title[stability of traveling wave with nonlocal dispersal]
{Stability of traveling waves in non-cooperative systems with nonlocal dispersal of equal diffusivities}

\author[J.-S. Guo]{Jong-Shenq Guo}
\address[J.-S. Guo]{Department of Mathematics, Tamkang University, Tamsui, New Taipei City 251301, Taiwan}
\email{jsguo@mail.tku.edu.tw}


\author[M. Shimojo]{Masahiko Shimojo}
\address[M. Shimojo]{Department of Mathematical Sciences, Tokyo Metropolitan University, Hachioji, Tokyo 192-0397, Japan}
\email{shimojo@tmu.ac.jp}

\thanks{Date: \today. Corresponding author: M. Shimojo}

\thanks{This work was supported in part by the National Science and Technology Council of Taiwan under the grant 112-2115-M-032-001 (JSG) and
by JSPS KAKENHI Grant-in-Aid for Scientific Research (C) (No.~	24K06817).}



\begin{abstract}
In this work, we first prove a stability theorem for traveling waves in a class of non-cooperative reaction-diffusion systems with nonlocal dispersal of equal diffusivities. 
Our stability criterion is in the sense that the initial perturbation is such that a suitable weighted
relative entropy function is bounded and integrable.
Then we apply our main theorem to derive the stability of traveling waves for some specific examples of non-cooperative systems arising in ecology and epidemiology.
\end{abstract}

\maketitle

\medskip

\section{Introduction}

It is well-known that, in contrast to the classical random diffusion, reaction-diffusion systems with nonlocal dispersal can better model
the long range movements and nonadjacent interactions of individuals.
This can be seen from many applied science models arising in physics, material science, population dynamics and so on (cf. \cite{hopf,fife, peter, Lutscher}).
In fact, when the dispersal kernel is highly concentrated, it is known that the model with nonlocal dispersal tends to the classical diffusion model.
Moreover, it is noted in \cite{al} that the dynamics of models with nonlocal dispersal is quite rich.
Therefore, the study of models with nonlocal dispersal has attracted a lot of attention in recent years.
However, there are certain difficulties arise in the study of nonlocal dispersal models.
One of them is that there is no regularizing effect for the nonlocal dispersal model in contrast to the classical diffusion case \cite{and,ir07}.

This work is devoted to the study of the following reaction-diffusion system with nonlocal dispersal:
\begin{equation}\label{RD}
(u_i)_t(x,t)=d_i\mathcal{N}_i[u_{i}](x,t)+u_i(x,t)f_i(u(x,t)),\;x\in\R,\, t>0,\; i=1,\dots,m,
\end{equation}
where $u:=(u_1,\dots,u_m)$, $d_i>0$, $f_i\in C^1([0,\infty)^m)$, $i=1,\dots,m$, $m$ is a positive integer and
\beaa
\mathcal{N}_i[u_i](x,t):=\int_R J_i(y)u_i(x-y,t)dy-u_i(x,t)=(J_i*u_i-u_i)(x,t),
\eeaa
in which the kernel function $J_i$ satisfies the following properties:
\begin{enumerate}
\item[(J1)] The kernel $J_i$ is nonnegative symmetric (w.r.t. $x=0$) and smooth in $\bR$;

\item[(J2)] it holds that
$$\int_{\mathbb{R}}J_i(y)dy=1;$$

\item[(J3)] it holds that $\int_{\mathbb{R}}J_i(y)e^{\lambda y} dy < \infty$ for all $\lambda \in (0, \hat\lambda_i)$ and
$$\int_{\mathbb{R}}J_i(y)e^{\lambda y} dy\to\infty\text{ as }\lambda \uparrow \hat\lambda_i$$
for some $\hat\lambda_i\in (0,\infty]$.
\end{enumerate}

System \eqref{RD} arises in many applications, such as population dynamics in ecology and epidemiology.
In population dynamics, we are particularly interested in the propagation phenomena of species in the biological model.
More precisely, it describes the invasion of certain species to other species in ecological systems (cf. \cite{fb00,ow01}),
or the spreading of certain diseases in epidemic models (cf. \cite{dot79,mk03}).
Among many different approaches towards the propagation phenomena, the existence vs non-existence of traveling waves and the spreading dynamics of solutions with
localized initial data are two most important subjects to be explored.
Although there are some well-known difficulties in the study of nonlocal dispersal models, some abstract theory from dynamical systems can be applied
to derive propagation properties for nonlocal models when it is of cooperative type.
For the theory and application of monotone semiflow to derive the minimal wave speed of traveling waves and the spreading speed,
we refer the reader to, e.g., \cite{ar1,wein,lui1989,weinberger,liang,ll08,jz09,fz14} and the references cited therein.

In this work, we are mainly concerned with traveling wave solutions of \eqref{RD} connecting
two {different constant equilibria $\{E^\pm\}$.} 
More precisely, a  traveling wave solution of \eqref{RD} is a solution in the form
\beaa\label{fw0}
u_i(x,t)=\phi_i(z),\; z:=x-ct,\; i=1,\dots,m,
\eeaa
with unknown function $\Phi:=(\phi_1,\dots,\phi_m)$ (the wave profile)
and unknown {positive constant $c$ (the wave speed)} satisfying
\be\label{bc}
{\Phi(-\infty)=E^-,\;\Phi(\infty)=E^+.}
\ee
The existence vs non-existence of traveling waves for systems with nonlocal dispersal has been
studied quite extensively in past years. 
We refer the reader to, e.g.,
\cite{coville08, coville13} for scalar equations, \cite{zlw12,zll09,dg19,ylw20,ylw21,zyl22,tyj} for predator-prey systems
and \cite{yllw13,ly14,lly14,ylw15,yl18} for epidemic models.

The main goal of this work is to derive the stability of traveling wave solutions for system \eqref{RD}.
Therefore, we shall always assume that \eqref{RD} admits positive traveling wave solutions 
$\{c,\Phi\}$ for all $c\ge c^*$ for some positive constant $c^*$. 
Hereafter $\Phi$ is positive means $\phi_i>0$ in $\R$ for all $i$.
The stability of traveling wave solutions in cooperative systems with nonlocal dispersal can be derived by sandwich method using the order-preserving property
of the cooperative systems. See, e.g., \cite{yp18}.
For non-cooperative systems with nonlocal dispersal, due to the lack of comparison principle, little is done for the stability of traveling waves.
In this work, motivated by \cite{gs23}, we provide a simple approach to tackle the stability of traveling waves for non-cooperative systems
with nonlocal dispersal.
Due to some technical difficulty, we shall only consider the equal diffusivities case in this work.
Hereafter, we shall assume that $d_i=1$ and $J_i=J$, $i=1,\cdots,m$, for some kernel $J$ satisfying (J1)-(J3) with $\hat\lambda\in(0,\infty]$.
Also, we set $\mathcal{N}[u_i]:=J*u_i-u_i$.

To study the stability of traveling waves, it is more convenient to use the so-called 
moving coordinate $z=x-ct$. Hence, for a positive solution $u$ of \eqref{RD}, 
the corresponding function $\{u_i=u_i(z,t)\}$ in terms of $z$-coordinate satisfies
\be\label{RDz}
(u_i)_t=\mathcal{N}[u_i]+c(u_i)_z +u_i f_i(u),\;z\in\R,\, t>0,\; i=1,\dots,m.
\ee
Note that a traveling wave $\{c,\Phi\}$ satisfies
\be\label{TWS}
\mathcal{N}[\phi_{i}]+c\phi_i'(z)+\phi_if_i(\Phi(z))=0,\;z\in\R,\; i=1,\dots,m,
\ee
hence $\Phi$ is a stationary solution of \eqref{RDz} for the given wave speed $c$.

For a given set of positive constant $\{\sigma_i\mid 1\le i\le m\}$ and 
a positive function $\Psi(z)=(\psi_1(z),\dots,\psi_m(z))$, $z\in\R$, 
we define as in \cite{gs23} the following relative entropy function
\be\label{mK}
\mK[\Psi](z):=\sum_{i=1}^m \sigma_i \mK_{i}[\psi_i](z),\quad 
\mK_{i}[\psi_{i}](z):=\psi_i(z)-\phi_i(z)-\phi_i(z)\ln\frac{\psi_i(z)}{\phi_i(z)}.
\ee
It is easy to see that $\mK[\Psi]\ge 0$ in $\R$ and $\mK[\Psi](z_0)=0$ if and only if $\Psi(z_0)=\Phi(z_0)$.
In particular, we denote the relative entropy function of $u(\cdot,t)$ by
\be\label{en}
W(z,t):=\mK[u(\cdot,t)](z),\; z\in\R,\, t>0.
\ee
Also, for a given positive constant $R$, we define the quantity
\be\label{cR}
c_R:=\inf_{0<\lambda<\hat\lambda}G(\lambda),\quad G(\lambda):=\frac{\left[ \int_{\mathbb{R}}J(y)e^{\lambda y}dy-1\right] + R}{\lambda }.
\ee
Note that the quantity $c_R$ is well-defined such that $c_R>0$.
This can be easily seen by the fact that $G(0^+)=G({\hat\lambda}^-)=\infty$ and $G(\lambda)>0$ for $\lambda\in(0,\hat\lambda)$.

With the above notation, we now state our main theorem on the stability of traveling waves
for system \eqref{RD} as follows.

\begin{theorem}\label{th:general}
{Let $R$ be a positive constant such that $c_R\ge c^*$,
$\{c,\Phi\}$ be a positive traveling wave solution of \eqref{RD} for some $c\ge c_R$,
and $u$ be a solution of \eqref{RDz} with positive continuous initial data $u_0$ at $t=0$.}
Suppose that there exists a set of positive constant $\{\sigma_i\mid 1\le i\le m\}$ such that
\be\label{sub}
W_t \le \mathcal{N}[W]+cW_z + RW,\; z\in\R,\, t>0,
\ee
for the relative entropy function $W$ of $u$ defined by \eqref{en}.
{Let $\lambda_{c}<\hat\lambda$ be the smallest positive root of $G(\lambda)=c$.}
Then, under the condition $e^{\lambda_{c} z}\mK[u_0]\in L^1(\R)\cap L^\infty(\R)$, $u(z,t)\to \Phi(z)$ as $t\to\infty$ locally uniformly for $z\in\R$.
\end{theorem}

The stability provided in Theorem~\ref{th:general} for traveling wave of \eqref{RD} is in the sense
that the initial perturbation is such that $e^{\ld_c z}\mathcal{E}[u_0]\in L^1(\R)\cap L^\infty(\R)$ at $t=0$.
Although the proof of Theorem~\ref{th:general} follows from the same idea as that in \cite{gs23} for the case of standard local diffusion,
there are some difficulties to be overcome due to the nonlocal dispersal.
With a suitable weight, thanks to a delicate estimate derived in \cite{ir07} (see Proposition~\ref{key1} below),
we are able to derive the convergence of the weighted relative entropy function to zero as $t\to\infty$.
Another key of the proof of this stability theorem is to derive the inequality \eqref{sub} with
an appropriate chosen set of positive constants $\{\sigma_i\}$.
Unfortunately, as in the case of classical diffusion, our method can only be applied for systems of nonlocal dispersal with equal diffusivities.
Systems with non-equal diffusivities are still left for open.

The rest of this paper is organized as follows. First, we provide a very simple proof of Theorem~\ref{th:general} along with a general calculation towards \eqref{sub} in \S2.
Then, in \S3, we provide an application of Theorem~\ref{th:general} to various non-cooperative systems studied in
\cite{dg19,ylw21,zyl22,tyj} for predator-prey systems and {in \cite{yllw13,yl18} for an epidemic model.}

\section{Proof of Theorem~\ref{th:general}}

First, we consider the linear problem
\be\label{linear}
\bss
v_t(x,t)=\mathcal{N}[v](x,t),\; x\in\bR,\, t>0,\\
v(x,0)=v_0(x),\; x\in\bR.
\ess
\ee
Recall from \cite[Theorem 1.3]{ir07} that

\begin{proposition}\label{key1}
Let {\rm (J1)-(J3)} be enforced for the kernel function $J$.
Suppose that $v_0\in L^1(\bR)\cap L^\infty(\bR)$. Then the solution $v$ of \eqref{linear} satisfies
the decay estimate
\be\label{est1}
\|v(\cdot,t)\|_{L^\infty(\bR)}\le C(1+t)^{-1/2},\; t\ge 0,
\ee
for some positive constant $C$ depending only on $\|v_0\|_{L^1(\bR)}$ and $\|v_0\|_{L^\infty(\bR)}$.
\end{proposition}

\begin{proof}[Proof of Theorem~\ref{th:general}]
Note that $w(z)=e^{-\lambda_{c} z}$ satisfies
\[
\mathcal{N}[w]+cw_z + Rw =0,\; z\in \R
\]
Define a function $V(z,t)= e^{\lambda_{c} z}W(z,t)$. Then, by \eqref{sub}, we obtain
\begin{align*}
V_{t}(z,t)&\le \int_{\mathbb{R}} J(y)e^{\lambda_{c} y}
V(z-y,t)\,dy +(R-1 -c \lambda_{c})V(z,t) +cV_{z}(z,t)\\
&=
\int_{\mathbb{R}} J(y)e^{\lambda_{c} y}
\{V(z-y,t)-V(z,t)\}\,dy  +cV_{z}(z,t).
\end{align*}
If we define a function $U(z+ct,t)=V(z,t)$, then
\[
U_{t}(x,t)\le
\int_{\mathbb{R}}
J(y)e^{\lambda_{c}y}
\{U(x-y,t)-U(x,t)  \}\,dy,\qquad x=z+ct.
\]

Now we introduce a new time variable $\tau$ by the relation
\[
\frac{\tau}{t}:=\int_{\mathbb{R}}J(y)e^{\lambda_{c}y}\,dy \in(0,\infty)
\]
to obtain
\be\label{U-tau}
U_{\tau}\le \frac{J(y)e^{\lambda_{c}y}}{\|J(y)e^{\lambda_{c}y}\|_{L^{1}}} * U  -U.
\ee
Note that the kernel function in \eqref{U-tau} satisfies (J1)-(J3).
Hence the comparison principle for scalar equations and \eqref{est1} imply that $U(x,\tau)\to 0$ as $\tau\to\infty$ uniformly for $x\in\R$.
Returning to the original variables $(z,t)$, Theorem~\ref{th:general} is thereby proved.
\end{proof}

{In order to apply Theorem~\ref{th:general} to some specific systems,
we first assume that system \eqref{RD} has {an invariant} set $\mI\subset [0,\infty)^m$.
Note that any nonnegative nontrivial solution of \eqref{RD} is positive, by applying the strong maximum principle for scalar equation.}

Secondly, we assume that 
\begin{equation}\label{ci}
I:=\sum_{i=1}^{m} \sigma_{i}\left(u_{i}-\phi_{i}\right) \{f_i(u)-f_i(\Phi)\} \le 0\quad\mbox{ for }\, u,\Phi\in \mI,
\end{equation}
for a given set of positive constants $\{\sigma_i\}$.
Following \cite{gs23}, we now perform a derivation of \eqref{sub} as follows.
By the definition of $W$ and using \eqref{RDz}, we have
\beaa
W_{z}&=&\sum_{i=1}^{m}\sigma_{i}\Big\{ (u_{i})_{z} \Big(
1-\frac{\phi_{i}}{u_{i}}\Big)-\phi_{i}' \ln \frac{u_{i}}{\phi_{i}}\Big\}\\
W_{t}&=&\sum_{i=1}^{m}\sigma_{i}
\Big(1-\frac{\phi_{i}}{u_{i}}\Big)(u_{i})_{t}\\
&=&\sum_{i=1}^{m}\sigma_{i}
\Big\{(u_{i}-\phi_{i})f_i(u)+\frac{u_{i}-\phi_{i}}{u_{i}}
\mathcal{N}[u_{i}] +\Big(1-\frac{\phi_{i}}{u_{i}}\Big)c (u_{i})_{z} \Big\}\\
&=&\sum_{i=1}^{m}\sigma_{i}
\Big\{(u_{i}-\phi_{i})(
f_i(u)-f_i(\Phi) )
+\frac{u_{i}-\phi_{i}}{u_{i}}
\mathcal{N}[u_{i}]+(u_{i}-\phi_{i})f_i(\Phi)\Big\}\\
&&\qquad\quad + c\sum_{i=1}^{m}\sigma_{i}{ \phi_{i}' \ln \frac{u_{i}}{\phi_{i}} }+cW_{z}.
\eeaa
Set $W_{i}=\mathcal{E}_{i}[u_{i}]$. Then
$\mathcal{N}[W_i]=\mathcal{N}[u_{i}]-\mathcal{N} [\phi_{i} ]-\mathcal{N}[\phi_{i} \ln{(u/\phi_{i})}]$.
Hence we obtain
\begin{eqnarray*}
&&W_{t}-\mathcal{N}[W]-cW_{z}\\
&=&\sum_{i=1}^{m}\sigma_{i}
\Big\{
(u_{i}-\phi_{i})(f_i(u)-f_i(\Phi))
-\frac{\phi_{i}}{u_{i}}\mathcal{N}[u_{i}]+f_i(\Phi)(u_{i}-\phi_{i})\\
&&\qquad\quad +\mathcal{N}[\phi_{i}]+\mathcal{N}[\phi_{i} \ln ({u_{i}}/{\phi_{i}})]
+ c \phi_{i}' \ln \frac{u_{i}}{\phi_{i}}\Big\}\\
&=&\sum_{i=1}^{m}\sigma_{i}
\Big\{(u_{i}-\phi_{i})(f_i(u)-f_i(\Phi))-\frac{\phi_{i}}{u_{i}}\mathcal{N}[u_{i}]+f_i(\Phi)
\Big(W_{i}+\phi_{i} \ln \frac{u_{i}}{\phi_{i}}\Big)\\
&&\qquad\quad +\mathcal{N}[\phi_{i}]+\mathcal{N}[\phi_{i} \ln ({u_{i}}/{\phi_{i}})]
+ c \phi_{i}' \ln \frac{u_{i}}{\phi_{i}} \Big\}.
\end{eqnarray*}

By substituting
\begin{align*}
-\frac{\phi_{i}}{u_{i}}\mathcal{N}[u_{i}]
&=-\frac{\phi_{i}}{u_{i}}J*u_{i}+\phi_{i},\\
\mathcal{N}\phi_{i}&=J*\phi_{i} -\phi_{i},\\
f_i(\Phi)\phi_{i}\ln \frac{u_{i}}{\phi_{i}}
&=-(J*\phi_{i} -\phi_{i} +c\phi_{i}' )\ln \frac{u_{i}}{\phi_{i}},\\
\mathcal{N}[\phi_{i} \ln (u_{i}/\phi_{i})]&
=J*(\phi_{i} \ln (u_{i}/\phi_{i}))-\phi_{i} \ln \frac{u_{i}}{\phi_{i}},
\end{align*}
{and using \eqref{ci}, we deduce} that
\beaa
&&W_{t}-\mathcal{N}[W]-cW_{z}- \sum_{i=1}^{m}{\sigma_i} f_i(\Phi) W_{i} \\
&\le&\sum_{i=1}^{m}\sigma_{i}
\Big\{-\frac{\phi_{i}}{u_{i}}J*u_{i}+J*\phi_{i}
-(J*\phi_{i})\ln \frac{u_{i}}{\phi_{i}}
+J* \Big(\phi_{i} \ln \frac{u_{i}}{\phi_{i}} \Big)\Big\}.
\eeaa
We further compute that
\beaa
&&-\frac{\phi_{i}}{u_{i}}J*u_{i}+J*\phi_{i}
-(J*\phi_{i})\ln \frac{u_{i}}{\phi_{i}}
+J* \Big(\phi_{i} \ln \frac{u_{i}}{\phi_{i}} \Big)\\
&=&\int_{\mathbb{R}}J(z-y)\Big\{
\phi_{i}(y)-\frac{u_{i}(y,t)}{u_{i}(z,t)}\phi_{i}(z)
-\phi_{i}(y)\ln \frac{u_{i}(z,t)}{\phi_{i}(z)}
+\phi_{i}(y) \ln \frac{u_{i}(y,t)}{\phi_{i}(y)}
\Big\}\,dy\\
&=&
\int_{\mathbb{R}}J(z-y)
\Big\{\phi_{i}(y)-\frac{u_{i}(y,t)}{u_{i}(z,t)}\phi_{i}(z)
+\phi_{i}(y) \ln \frac{u_{i}(y,t)\phi_{i}(z)}{u_{i}(z,t)\phi_{i}(y)}
\Big\}\,dy\\
&\le&
\int_{\mathbb{R}}J(z-y)
\Big\{\phi_{i}(y)-\frac{u_{i}(y,t)}{u_{i}(z,t)}\phi_{i}(z)
+\phi_{i}(y) \Big(\frac{u_{i}(y,t)\phi_{i}(z)}{u_{i}(z,t)\phi_{i}(y)}
-1\Big)\Big\}\,dy=0.
\eeaa
Here we used the inequality $\ln X \le X-1$ for $X>0$, by setting
$$X=\frac{u(y,t)\phi_{i}(z)}{u(z,t)\phi_{i}(y)}.$$

Finally, if we also have
\begin{equation*}
\max_{1 \le i \le m} \{ \|f_i(\Phi)\|_{L^\infty(\R)}\}\le R,
\end{equation*}
then we can conclude from \eqref{ci} that \eqref{sub} holds with this $R$ and
so Theorem~\ref{th:general} is applicalbe.

\section{Application of Theorem~\ref{th:general}}
\setcounter{equation}{0}

In this section, we apply Theorem~\ref{th:general} to derive the {stability of traveling waves in some non-cooperative systems arising
in ecology and epidemiology

\subsection{Predator-prey models}\hfil\\


First, in \cite{dg19}, we consider
\begin{equation}\label{pp}
\begin{cases}
(u_1)_t(x,t)=\mathcal N[u_1](x,t)+r_1u_1(x,t)[1-u_1(x,t)-au_2(x,t)],\; x\in\R,\, t>0,\\
(u_2)_t(x,t)=\mathcal N[u_2](x,t)+r_2u_2(x,t)[-1+bu_1(x,t)-u_2(x,t)],\; x\in\R,\, t>0,
\end{cases}
\end{equation}
where $r_1,r_2,a,b$ are positive constants.
We assume
\begin{equation}\label{constant}
b>1,\;\;  ab<1.
\end{equation}
Then the quantity
\begin{equation}\label{cstar}
c^{\ast }:=\inf_{0<\lambda<\hat\lambda}\frac{\left[ \int_{\mathbb{R}}J(y)e^{\lambda y}dy-1\right] + r_2(b-1)}{\lambda }
\end{equation}
is well-defined and $c^*>0$.

For the existence of traveling waves, we recall from \cite{dg19,w21} that
system \eqref{pp} admits a traveling wave solution $\{c,(\phi_1,\phi_2)\}$ satisfying \eqref{bc} with
\beaa
E^+=(1,0),\quad E^-=\left(\frac{1+a}{1+ab},\frac{b-1}{1+ab}\right),
\eeaa
for any $c>c^\ast$; while such a traveling wave exists for $c=c^*$ if we further assume that $J$ is compactly supported.



Note that, by the comparison for the scalar equation, $\mathcal{I}:=[0,1]\times[0,b-1]$ is an invariant set of system \eqref{pp}.
We choose $\sigma_1=1/r_1$ and $\sigma_2=a/(r_2b)$. Then the quantity $I$ in \eqref{ci} is computed as
\beaa
I=-(u_1-\phi_1)^2-\frac{a}{b}(u_2-\phi_2)^2\le 0,\;\forall\, (u_1,u_2),(\phi_1,\phi_2)\in\mathcal{I}.
\eeaa
Hence \eqref{sub} holds with $R:=\max\{r_1,r_2(b-1)\}$.
We conclude from Theorem~\ref{th:general} that the stability of traveling waves for \eqref{pp} in the sense described in Theorem~\ref{th:general}
holds for all $c\ge c_{R}$. It holds for all $c\ge c^*$, if we further assume that $r_1\le r_2(b-1)$.

Secondly, for the predator-prey system with two weak competing predators $u_1,u_2$ and one prey $u_3$ with the nonlinearities in \eqref{RD} defined by
\beaa
\bss
f_1(u_1,u_2,u_3)=r_1(-1-u_1-hu_2+bu_3),\\
f_2(u_1,u_2,u_3)=r_2(-1-ku_1-u_2+bu_3),\\
f_3(u_1,u_2,u_3)=r_3(1-au_1-au_2-u_3),
\ess
\eeaa
where $r_1,r_2,r_3>0$, $b>1$, $0<a<1/[2(b-1)]$ and $0<h,k<1$,
the existence of traveling waves connecting
the predator-free state $(0,0,1)$ and the unique positive coexistence state were obtained in \cite{gnow20} for the case of classical diffusion;
and in \cite{ylw21} for the case of nonlocal dispersal. Note that the existence of waves is not restricted to the case of equal diffusivities.
Applying Theorem~\ref{th:general} with the same choice of $\{\sigma_i\}$ as in \cite{gs23}, the stability
with initial perturbation in $L^1(\R)\cap L^\infty(\R)$ for traveling waves of the nonlocal dispersal case with equal diffusivities can be derived.
Here we have $R=\max\{r_1(b-1),r_2(b-1),r_3\}$, {since $0\le\phi_3\le 1$ and $0\le\phi_1,\phi_2\le b-1$. Moreover, we have}
\beaa
c^*:=\inf_{0<\lambda<\hat\lambda}\frac{\left[ \int_{\mathbb{R}}J(y)e^{\lambda y}dy-1\right] + \max\{r_1,r_2\}(b-1)}{\lambda},
\eeaa
and the stability holds for any wave with speed $c\ge c_R$. Note that $c_R=c^*$, if we assume that $r_3\le \max\{r_1,r_2\}(b-1)$.
Since it is quite similar to \cite[Theorem 3.7]{gs23}, we safely omit the further details here.

Similarly, for the predator-prey system with two weak competing preys $u_1,u_2$ and one predator $u_3$
with the nonlinearities in \eqref{RD} defined by
\beaa
\bss
f_1(u_1,u_2,u_3)=r_1(1-u_1-hu_2-au_3),\\
f_2(u_1,u_2,u_3)=r_2(1-ku_1-u_2-au_3),\\
f_3(u_1,u_2,u_3)=r_3(-1+bu_1+bu_2-u_3),
\ess
\eeaa
where $r_1,r_2,r_3,a,b>0$ (with some further restrictions on $a,b$) and $0<h,k<1$, we can obtain a stability result
described in Theorem~\ref{th:general} for the traveling waves connecting the predator-free state to the unique positive coexistence state
derived in \cite{tyj} for the nonlocal dispersal system with equal diffusivities.
Here the predator-free state is $(u_p,v_p,0)$, where
\beaa
u_p:=\frac{1-h}{1-hk}\in(0,1),\quad v_p:=\frac{1-k}{1-hk}\in(0,1),
\eeaa
the critical wave speed $c^*$ is defined by
\beaa
c^*:=\inf_{0<\lambda<\hat\lambda}\frac{\left[ \int_{\mathbb{R}}J(y)e^{\lambda y}dy-1\right] + r_3[b(u_p+v_p)-1]}{\lambda},
\eeaa
and the constant $R=\max\{r_1,r_2,r_3(2b-1)\}$, {using $0\le\phi_1,\phi_2\le 1$ and $0\le\phi_3\le 2b-1$}.
Similar to \cite{gs23}, we can only obtain the stability result described in Theorem~\ref{th:general} for the nonlocal dispersal system with equal diffusivities for any wave with speed {$c\ge c_R$}, {where $c_R>c^*$, since $R\ge r_3(2b-1)>r_3[b(u_p+v_p)-1]$.}

Thirdly, in \cite{zyl22}, they considered a predator-prey system with two preys (without inter-specific competition between preys)
and one predator with/without intra-specific competition, namely, the nonlinearities in \eqref{RD} are defined by
\beaa
\bss
f_1(u_1,u_2,u_3)=r_1(1-u_1-a_1u_3),\\
f_2(u_1,u_2,u_3)=r_2(1-u_2-a_2u_3),\\
f_3(u_1,u_2,u_3)=r_3(-1+b_1u_1+b_2u_2-\gamma u_3),
\ess
\eeaa
where $r_1,r_2,r_3,a_1,a_2,b_1,b_2$ are positive constants with $b_1+b_2>1$ and $\gamma$ is a nonnegative constant to denote whether there is the intra-specific competition in the predator.
It is clear that $\mathcal{I}:=[0,1]\times[0,1]\times[0,b_1+b_2-1]$ is an invariant set of this predator-prey system.
By choosing
\beaa
\sigma_1=\frac{1}{r_1},\; \sigma_2=\frac{a_1b_2}{r_2a_2b_1},\; \sigma_3=\frac{a_1}{r_3b_1},
\eeaa
the quantity $I$ in \eqref{ci} can be computed as
\beaa
I=-(u_1-\phi_1)^2-\frac{a_1b_2}{a_2b_1}(u_2-\phi_2)^2-\gamma\frac{a_1}{b_1}(u_3-\phi_3)^2\le 0,\;\forall\, (u_1,u_2,u_3),(\phi_1,\phi_2,\phi_3)\in\mathcal{I}.
\eeaa
Hence \eqref{sub} holds with $R:=\max\{r_1,r_2,r_3(b_1+b_2-1)\}$.
Therefore, the stability of traveling waves connecting the predator-free state $(1,1,0)$ and the unique positive coexistence state with speed $c\ge c_R$
for this predator-prey system with nonlocal dispersal of equal diffusivities (cf. \cite{zyl22} for the {existence of traveling waves})
follows from Theorem~\ref{th:general}.
The stability holds for any $c\ge c^*$, where
\beaa
c^*:=\inf_{0<\lambda<\hat\lambda}\frac{\left[ \int_{\mathbb{R}}J(y)e^{\lambda y}dy-1\right] + r_3(b_1+b_2-1)}{\lambda},
\eeaa
if we further assume that $r_i\le r_3(b_1+b_2-1)$ for $i=1,2$.

\subsection{An epidemic model}\hfil\\

{Finally, we present in this section a non-cooperative system arising in epidemiology as follows.}
In \cite{yllw13,yl18}, they considered a nonlocal dispersal Kermack-McKendrick epidemic model
described by \eqref{RD} with
\beaa
f_1(u_1,u_2)=-\beta u_2,\; f_2(u_1,u_2)=\beta u_1-\gamma,
\eeaa
where $u_1$ is the susceptible population, $u_2$ is the infective population,
and $\beta,\gamma$ are positive constants which stand for the infection rate and the removal rate, respectively.
Note that this system has a family of constant stationary solutions $\{(s,0)\mid s>0\}$.

Let $s^{*}>0$ be a fixed constant such that $\beta s^{*}>\gamma$.
Then, under the assumption that $J$ is compactly supported, the existence of traveling wave connecting $(s^*,0)$ and $(s_0,0)$
for some $s_0\le s^*$ for system \eqref{RD} with equal diffusivities was obtained {in} \cite{yllw13} for $c\in(0,c^*)$ and {in}
\cite{yl18} for $c=c^*$, where the quantity
\begin{equation*}
c^{\ast }:=\inf_{0<\lambda<\infty}\frac{\left[ \int_{\mathbb{R}}J(y)e^{\lambda y}dy-1\right] + \beta s^{*}-\gamma}{\lambda }
\end{equation*}
is well-defined and $c^*>0$.

It is easy to see that $[0,s^*]\times[0,\infty)$ is an invariant set of this epidemic model.
By choosing $\sigma_1=\sigma_2=1$, the quantity $I$ in \eqref{ci} {is identically equal} to zero.
Hence \eqref{sub} holds with $R=\beta s^{*}-\gamma$.
We conclude from Theorem~\ref{th:general} that the stability of these traveling waves in the sense described in Theorem~\ref{th:general}
holds for all $c\ge c^{*}$.



\bigskip

\noindent
{\bf Declarations of interest: none.}

\bigskip

\end{document}